\pdfoutput=1
\documentclass[microtype]{gtpart}
\usepackage{microtype}
\usepackage[dvipsnames]{xcolor}
\usepackage{enumerate, multicol}
\usepackage{graphicx,tikz}
\usetikzlibrary{decorations.markings} 
\usepackage[mathscr]{eucal}
\usepackage{amssymb}
\usepackage[abbrev]{amsrefs}
\usepackage{xcolor}

\newcommand{\rr}{\mathbb{R}}

\newcommand{\zz}{\mathbb Z}
\newcommand{\bn}{\mathbb N}
\newcommand{\bS}{\mathbb S}

\newcommand{\cM}{\mathcal M}
\newcommand{\cP}{\mathcal P}

\newcommand{\gS}{\Sigma}

\newcommand{\ssm}{\smallsetminus}

\newcommand{\inv}{^{-1}}

\DeclareMathOperator{\Mod}{Mod}
\DeclareMathOperator{\Hom}{Hom}
\DeclareMathOperator{\Homeo}{Homeo}

\DeclareMathOperator{\Aut}{Aut}

\newtheorem*{mainthm}{Main Theorem}
\newtheorem{Theorem}{Theorem}[section]
\newtheorem{Theorem*}{Theorem}

\newtheorem{lemma}[Theorem]{Lemma}

\newtheorem{Cor*}[Theorem*]{Corollary}

\newtheorem*{Remark}{Remark} 

\theoremstyle{definition}
\newtheorem{definition}[Theorem]{Definition}

\numberwithin{equation}{section}

\definecolor{torusgreen}{RGB}{1,199,34}
\definecolor{torusblue}{RGB}{0,48,231}
\definecolor{torusyellow}{RGB}{243,170,0}
\definecolor{toruspink}{RGB}{214,0,226}
\definecolor{torusred}{RGB}{221,5,0}

\title{Thurston norms of tunnel number-one manifolds}

\author{Natalia Pacheco-Tallaj}
\address{Department of Mathematics \\ Harvard College \\Cambridge, MA 02138\\USA}
\email{pachecotallaj@college.harvard.edu}

\author{Kevin Schreve}
\address{Department of Mathematics\\
University of Michigan\\   Ann Arbor, MI 48109\\USA}
\email{schreve@umich.edu}

\author{Nicholas G.~Vlamis}
\address{Department of Mathematics\\
University of Michigan\\   Ann Arbor, MI 48109\\USA}
\email{vlamis@umich.edu}



\begin{abstract} 
The Thurston norm of a $3$-manifold measures the complexity of surfaces representing two-dimensional homology classes. 
We study the possible unit balls of Thurston norms of $3$-manifolds $M$ with $b_1(M) = 2$, and whose fundamental groups admit presentations with two generators and one relator. We show that even among this special class, there are $3$-manifolds such that the unit ball of the Thurston norm has arbitrarily many faces. 
\end{abstract}

\begin{document}

\maketitle

\section{Introduction}\label{s:intro}

A \emph{knot} is the image of a smooth embedding of the circle \( \bS^1 \) into the 3-sphere \( \bS^3 \).
Every knot \( K \subset \bS^3 \) bounds an embedded orientable surface in \( \bS^3 - K \) called a \emph{Seifert surface}. 
This leads to a useful knot invariant, called the \emph{genus} of  \( K \), which is the minimal genus of such a Seifert surface. 
For example, the unknot is the only knot whose genus is zero. 
This invariant can be generalized to the \emph{Thurston norm}, which is defined for all compact, connected, orientable $3$--manifolds.

Specifically, for each \( \phi\in H^1(N,\mathbb{Z}) \), let $PD(\phi)$ denote the homology class in  \(H_2(N, \partial N, \mathbb{Z}) \) that is Poincar\'e dual to \( \phi \). 
This class can be represented by a properly embedded oriented surface \( \gS \). 
Let \( \chi_-(\gS)=\sum_{i=1}^k \max\{-\chi(\gS_i),0\} \), where  \( \gS_1, \ldots, \gS_k \) are the connected components of \( \gS \). The \emph{Thurston norm} of \( \phi \)  is defined to be: 
\[
x_N(\phi):=\min \big\{ \chi_-(\gS) | [\gS] = PD(\phi) \in H_2(N, \partial N; \mathbb{Z}) \}.
\]

Thurston showed that \( x_N \) can be extended to a seminorm on \( H^1(N,\rr) \). 
Since a Seifert surface of a knot \( K \) is dual to the generator of  \( H^1(\bS^3 - K, \zz) \) corresponding to a meridional loop, the Thurston norm of this generator is \( 2g(K) -3 \), where \( g(K) \) is the genus of \( K \); hence, the Thurston norm generalizes the knot genus.

The Thurston norm  is also useful in studying the possible ways that a $3$-manifold can fiber over a circle. 
Recall that a class \( \phi\in H^1(N;\mathbb{R}) \) is \emph{fibered} if  \( \phi \) is represented by a non-degenerate closed 1-form; in particular, \( \phi \in H_1(N, \zz)\) is fibered if and only if it is induced by a fibration \( N\to \bS^1 \).

The results of Thurston \cite[Theorems 1, 2 \& 3]{Th86} can be summarized in the following theorem.
A \emph{marked polytope} is a polytope with a distinguished subset of vertices, which we refer to as \emph{marked vertices}.

\begin{Theorem}[Thurston]\label{thm:thurston}
Let $N$ be a $3$-manifold. There exists a unique centrally symmetric marked polytope $\cM_N$ in $H_1(N;\rr)$ such that for any $\phi\in H^1(N;\rr) = \Hom(H_1(N,\rr),\rr)$ we have 
\[ 
x_N(\phi)=\max\{ \phi(p)-\phi(q) : p,q\in \cM_N\}. 
\]
Moreover, \( \phi \) is fibered if and only if \( \phi \) restricted to \( \cM_N \) attains its maximum on a marked vertex. 
\end{Theorem}

The polytope \( \cM_N \) is dual to the unit ball of the Thurston norm on \( H^1(N,\rr) \), see Section \ref{s:algo}. 
We let \( \cP_N \) denote \( \cM_N \) without the marked vertices. 
As a sample application, this immediately implies that if a 3-manifold fibers over a circle and the first Betti number \( b_1(N) \) is at least 2, then it fibers in infinitely many ways. 

We are interested in the possible (marked) polytopes that arise as \( \cM_N \) for a 3-manifold \( N \).
We first note some necessary conditions:  If \( \cM \) is the marked polytope of some 3-manifold, then 
\begin{enumerate}
\item \( \cM \) is centrally symmetric, bounded, and convex.
\item \( \cM \) can be translated by a vector \( v \in \frac{1}{2}H_1(N, \zz) \) so that its vertices are in \( H_1(N, \zz) / \text{torsion} \) (this follows from the fact that \( x_N \) assigns integral values to elements of \( H_1(N, \zz) \)).
\end{enumerate}

Thurston showed \cite[Corollary to Theorem 6]{Th86} that for any norm \( x \) on  \( \rr^2 \) that takes even integral values on \( \zz^2 \), there is a closed 3-manifold \( N \) such that \( x_N \)  is equal to \( x \). 
We note that Thurston's proof, while explicit,  gives no control over the complexity of the fundamental groups of the 3-manifolds constructed. 

Agol asked the second author whether any polygon in \( \rr^2 \) satisfying (1) and (2) above is \( \cM_N \) for a very simple 3-manifold \( N \); in particular, a 3-manifold whose fundamental group admits a \( (2,1) \)-presentation of the form  \( \pi_1(N) = \langle x,y | r \rangle \). 
We cannot give a complete answer, but our main theorem states that the restriction to such simple manifolds does not restrict the complexity of the associated norms:

\begin{mainthm}\label{t:main}
For any \( n \in \mathbb{N} \), there exists a 3-manifold \( N \) such that \( \pi_1(N) \) admits a \( (2,1) \)-presentation, \( \cM_N \) is a polygon with \( (2n) \)-sides, every vertex of \( \cM_N \) is marked, and \( \cM_N \) has \( \left \lfloor{\frac{n}{2}}\right \rfloor \) sides of different lengths.
\end{mainthm}

The manifolds in the Main Theorem are constructed by gluing a $2$-handle onto an orientable genus-two handlebody along a regular neighborhood of a homotopically non-trivial simple closed curve (these are usually called \emph{tunnel number one manifolds}). 
Since we require $H_1(N, \rr) \cong \rr^2$, we will always choose the curve to be separating. 
By Theorem \ref{thm:thurston}, every cohomology class in the examples from the main theorem is fibered.

To calculate the Thurston norm of these examples, we use an algorithm of Friedl, the second author, and Tillmann \cite{fst}, which gives an easy way to read off $\cM_N$ from the relator in a \( (2,1) \)-presentation of \( \pi_1(N) \). 
To construct a wide variety of examples, we consider the orbit of a particular separating curve under the mapping class group of the marked boundary surface of the handlebody, which gives us many possible curves to glue $2$-handles onto. 
Dunfield and D. Thurston have used this same construction with a random walk in the mapping class group to show that a random tunnel number one manifold does not fiber over a circle \cite[Theorem 2.4]{dt}. 
In this sense, our examples are non-generic.

In Section \ref{s:algo} we review the algorithm which computes the marked polytope $\cM_N$. In Section \ref{s:mcg}, we recall standard generators of the mapping class group and how they behave as automorphisms of the fundamental group of a surface.
In Section \ref{s:construction} we use these generators to derive a complicated relator $r$ which produces the desired complicated Thurston norm unit ball. 
In the appendix, we do an explicit calculation of the Thurston norm of a $3$-manifold with fundamental group the Baumslag-Solitar group $BS(m,m)$.

\subsection*{Acknowledgements}
We would like to thank Nathan Dunfield for providing us with useful Python code to simplify group presentations, as well as tables of knot group presentations. 
The second author would like to thank Ian Agol for the problem and Stefan Friedl for useful suggestions. 
This work was completed as part of the REU program at the University of Michigan, for the duration of which the first author was supported by NSF grants DMS-1306992 and DMS-1045119. 
We would like to thank the organizers of the Michigan REU program for their efforts.
The second and third authors were partially supported by NSF grant DMS-1045119. This material is based upon work supported by the National Science Foundation under Award No. 1704364. 


\section{Algorithm for the Thurston norm}\label{s:algo}

In \cite{fst}, the following algorithm was given for computing the Thurston norm of $3$-manifolds $N$ with $\pi_1(N) = \langle x,y|r\rangle$ and $b_1(N) = 2$, see also \cite{ft}.
The relator $r$ determines a walk on $H_1(N;\zz)$ in $H_1(N;\rr) \cong \rr^2$. We assume that $r$ is reduced and cyclically reduced. 
The marked polytope $\cM_M$ is constructed as follows (see Figure \ref{f:algorithm} for an example):

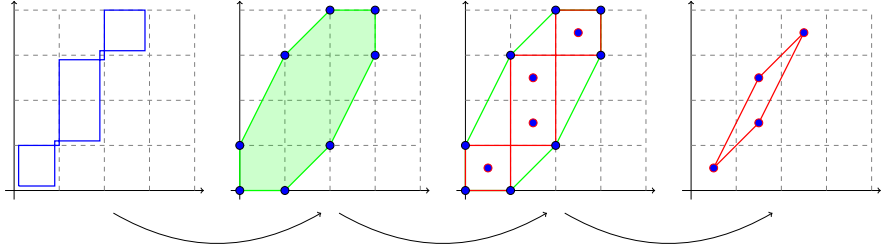
\begin{figure}
\centering
\scalebox{0.6}{
\begin{tikzpicture}
      \draw[->] (-0.2,0) -- (4.2,0);
      \draw[->] (0,-0.2) -- (0,4.2);
      \draw[gray, dashed] (1,0)--(1,4);
      \draw[gray, dashed] (2,0)--(2,4);
      \draw[gray, dashed] (3,0)--(3,4);
      \draw[gray, dashed] (4,0)--(4,4);
      \draw[gray, dashed] (0,1)--(4,1);
      \draw[gray, dashed] (0,2)--(4,2);
      \draw[gray, dashed] (0,3)--(4,3);
      \draw[gray, dashed] (0,4)--(4,4);
      \draw[thick, blue] 
          (0.1,0.1) 
       -- (0.9, 0.1)
       -- (0.9, 1.1)
       -- (1.9, 1.1)
       -- (1.9, 3.1)
       -- (2.9, 3.1)
       -- (2.9, 4)
       -- (2, 4)
       -- (2, 2.9)
       -- (1, 2.9)
       -- (1, 1)
       -- (0.1, 1)
       -- (0.1, 0.1);
      \draw[->] (4.8,0) -- (9.2,0);
      \draw[->] (5,-0.2) -- (5,4.2);
      \draw[gray, dashed] (6,0)--(6,4);
      \draw[gray, dashed] (7,0)--(7,4);
      \draw[gray, dashed] (8,0)--(8,4);
      \draw[gray, dashed] (9,0)--(9,4);
      \draw[gray, dashed] (5,1)--(9,1);
      \draw[gray, dashed] (5,2)--(9,2);
      \draw[gray, dashed] (5,3)--(9,3);
      \draw[gray, dashed] (5,4)--(9,4);
      \draw[thick, green]
          (5, 0)
       -- (6, 0)
       -- (7, 1)
       -- (8, 3)
       -- (8, 4)
       -- (7, 4)
       -- (6, 3)
       -- (5, 1)
       -- (5, 0);
      \fill[green,fill opacity=0.2]
          (5, 0)
       -- (6, 0)
       -- (7, 1)
       -- (8, 3)
       -- (8, 4)
       -- (7, 4)
       -- (6, 3)
       -- (5, 1)
       -- (5, 0);
      \node[circle,draw=black, fill=blue, inner sep=0pt,minimum size=5pt] (b) at (5, 0) {};
      \node[circle,draw=black, fill=blue, inner sep=0pt,minimum size=5pt] (b) at (6, 0) {};
      \node[circle,draw=black, fill=blue, inner sep=0pt,minimum size=5pt] (b) at (7, 1) {};
      \node[circle,draw=black, fill=blue, inner sep=0pt,minimum size=5pt] (b) at (8, 3) {};
      \node[circle,draw=black, fill=blue, inner sep=0pt,minimum size=5pt] (b) at (8, 4) {};
      \node[circle,draw=black, fill=blue, inner sep=0pt,minimum size=5pt] (b) at (7, 4) {};
      \node[circle,draw=black, fill=blue, inner sep=0pt,minimum size=5pt] (b) at (6, 3) {};
      \node[circle,draw=black, fill=blue, inner sep=0pt,minimum size=5pt] (b) at (5, 1) {};
      \draw[->] (9.8,0) -- (14.2,0);
      \draw[->] (10,-0.2) -- (10,4.2);
      \draw[gray, dashed] (11,0)--(11,4);
      \draw[gray, dashed] (12,0)--(12,4);
      \draw[gray, dashed] (13,0)--(13,4);
      \draw[gray, dashed] (14,0)--(14,4);
      \draw[gray, dashed] (10,1)--(14,1);
      \draw[gray, dashed] (10,2)--(14,2);
      \draw[gray, dashed] (10,3)--(14,3);
      \draw[gray, dashed] (10,4)--(14,4);
      \draw[thick, green]
          (10, 0)
       -- (11, 0)
       -- (12, 1)
       -- (13, 3)
       -- (13, 4)
       -- (12, 4)
       -- (11, 3)
       -- (10, 1)
       -- (10, 0);
      \draw[red, thick]
          (10, 0)
       -- (11, 0)
       -- (11, 3)
       -- (13, 3)
       -- (13, 4)
       -- (12, 4)
       -- (12, 1)
       -- (10, 1)
       -- (10, 0);
      \node[circle,draw=black, fill=blue, inner sep=0pt,minimum size=5pt] (b) at (10, 0) {};
      \node[circle,draw=black, fill=blue, inner sep=0pt,minimum size=5pt] (b) at (11, 0) {};
      \node[circle,draw=black, fill=blue, inner sep=0pt,minimum size=5pt] (b) at (12, 1) {};
      \node[circle,draw=black, fill=blue, inner sep=0pt,minimum size=5pt] (b) at (13, 3) {};
      \node[circle,draw=black, fill=blue, inner sep=0pt,minimum size=5pt] (b) at (13, 4) {};
      \node[circle,draw=black, fill=blue, inner sep=0pt,minimum size=5pt] (b) at (12, 4) {};
      \node[circle,draw=black, fill=blue, inner sep=0pt,minimum size=5pt] (b) at (11, 3) {};
      \node[circle,draw=black, fill=blue, inner sep=0pt,minimum size=5pt] (b) at (10, 1) {};
      \node[circle,draw=red, fill=blue, inner sep=0pt,minimum size=5pt] (b) at (10.5, 0.5) {};
      \node[circle,draw=red, fill=blue, inner sep=0pt,minimum size=5pt] (b) at (11.5, 1.5) {};
      \node[circle,draw=red, fill=blue, inner sep=0pt,minimum size=5pt] (b) at (11.5, 2.5) {};
      \node[circle,draw=red, fill=blue, inner sep=0pt,minimum size=5pt] (b) at (12.5, 3.5) {};
      \draw[->] (14.8,0) -- (19.2,0);
      \draw[->] (15,-0.2) -- (15,4.2);
      \draw[gray, dashed] (16,0)--(16,4);
      \draw[gray, dashed] (17,0)--(17,4);
      \draw[gray, dashed] (18,0)--(18,4);
      \draw[gray, dashed] (19,0)--(19,4);
      \draw[gray, dashed] (15,1)--(19,1);
      \draw[gray, dashed] (15,2)--(19,2);
      \draw[gray, dashed] (15,3)--(19,3);
      \draw[gray, dashed] (15,4)--(19,4);
      \draw[thick, red]
          (15.5, 0.5)
       -- (16.5, 1.5)
       -- (17.5, 3.5)
       -- (16.5, 2.5)
       -- (15.5, 0.5);
      \node[circle,draw=red, fill=blue, inner sep=0pt,minimum size=5pt] (b) at (15.5, 0.5) {};
      \node[circle,draw=red, fill=blue, inner sep=0pt,minimum size=5pt] (b) at (16.5, 1.5) {};
      \node[circle,draw=red, fill=blue, inner sep=0pt,minimum size=5pt] (b) at (16.5, 2.5) {};
      \node[circle,draw=red, fill=blue, inner sep=0pt,minimum size=5pt] (b) at (17.5, 3.5) {};
      \draw[->] (2.2, -0.5) to [out=330,in=210] (6.8, -0.5);
      \draw[->] (7.2, -0.5) to [out=330,in=210] (11.8, -0.5);
      \draw[->] (12.2, -0.5) to [out=330,in=210] (16.8, -0.5);
\end{tikzpicture}
}
\caption{ The algorithm applied to the group: $\langle x,y \,|\, r=xyxyyxyXYXYYXY\rangle$, where capital letters denote inverses. This is the first case of our sequence of examples in Section \ref{s:construction}. }
\label{f:algorithm}
\end{figure}

\begin{enumerate} 
\item Start at the origin and walk on $\zz^2$ reading the word $r$ from left to right.
\item Take the convex hull $\mathcal{C}$ of the lattice points reached by the walk.
\item Mark the vertices of $\mathcal{C}$ which the walk passes through exactly once.
\item Consider the unit squares that are contained in $\mathcal{C}$ and touch a vertex of $\mathcal{C}$. Mark a midpoint of a square precisely when a vertex of $\mathcal{C}$ incident with the square is marked.
\item The set of vertices of $\cM_N$ is the set of midpoints of these squares, and a vertex of $\cM_N$ is marked precisely when it is a marked midpoint of a square.
\end{enumerate}

To obtain the unit ball of the Thurston norm $x_{N}$, we consider the dual polytope
\[
\cM_N^\ast := \{\phi \in H^1(N,\rr) : \phi(v) - \phi(w) \le 1 \text{ for all } v,w \in \cM_N\}.
\]


\section{Mapping class groups}
\label{s:mcg}

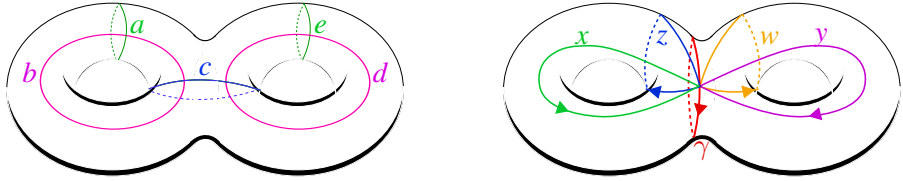
\begin{figure}
\centering
\begin{minipage}{.5\textwidth}
  \centering
  \begin{tikzpicture}
  	\node[anchor=south west, inner sep=0] (image) at (0,0) {\includegraphics[width=.8\linewidth]{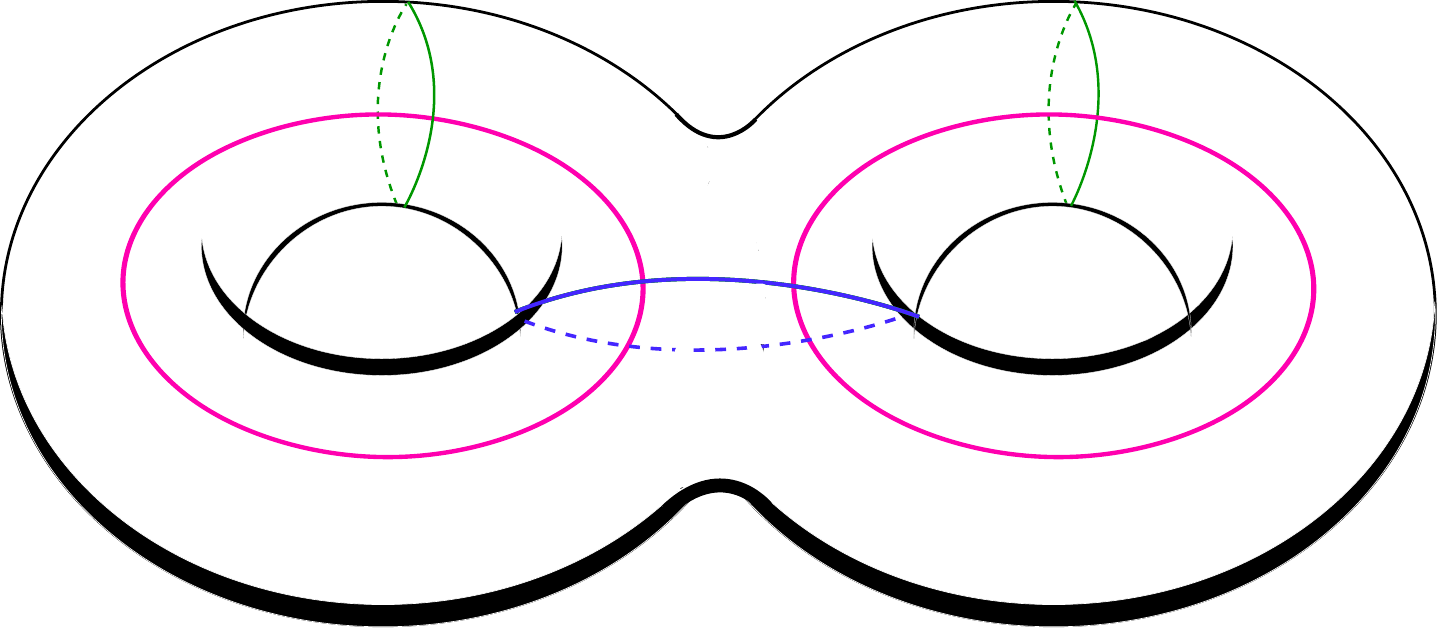}};
	\begin{scope}[
        x={(image.south east)},
        y={(image.north west)}
       ]
       		\node [torusgreen, font=\bfseries] at (0.33, 0.86) {$a$};
       		\node [torusgreen, font=\bfseries] at (0.79, 0.86) {$e$};
       		\node [toruspink, font=\bfseries] at (0.06, 0.6) {$b$};
       		\node [toruspink, font=\bfseries] at (0.94, 0.6) {$d$};
		\node [torusblue, font=\bfseries] at (0.5, 0.62) {$c$};
       \end{scope}
  \end{tikzpicture}
\end{minipage}%
\begin{minipage}{.5\textwidth}
  \centering
  \begin{tikzpicture}
      \node[anchor=south west,inner sep=0] (image) at (0,0) {\includegraphics[width=.8\linewidth]{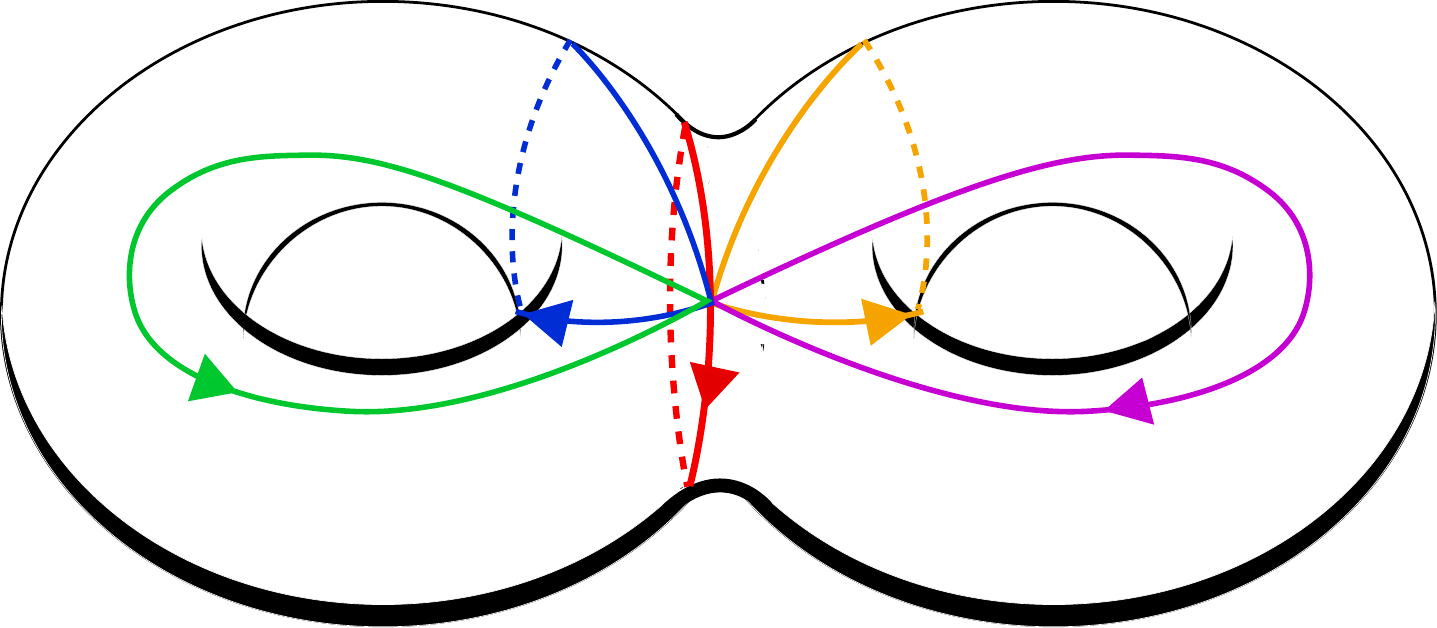}};
      \begin{scope}[
        x={(image.south east)},
        y={(image.north west)}
      ]
           \node [torusgreen, font=\bfseries] at (0.2,0.8) {$x$};
           \node [torusblue, font=\bfseries] at (0.4,0.8) {$z$};
           \node [torusyellow, font=\bfseries] at (0.67,0.8) {$w$};
           \node [toruspink, font=\bfseries] at (0.8,0.8) {$y$};
           \node [torusred, font=\bfseries] at (0.5,0.15) {$\gamma$};
      \end{scope}
  \end{tikzpicture}
\end{minipage}
\caption{The right Dehn twists about the curves \( a,b,c,d \), and \( e \) generate \( \Mod(\Sigma_2) \).  The loops \( w,x,y, \) and \( z \) are standard generators for \( \pi_1(\Sigma_2) \), which we will use throughout the article.  The curve \( \gamma \) in \( \pi_1(\Sigma_2) \) is the commutator \( [x,z] \).
}
\label{fig:t2twists}
\end{figure}

For background and details on mapping class groups, we refer the reader to \cite{fm}.
For our purposes, we focus on a single surface: the closed genus two surface \( \Sigma_2 \).
As we will be working with fundamental groups, we need to keep track of a basepoint.
Fix a point \( p \) in \( \Sigma_2 \).
Let \( \Homeo^+(\Sigma_2, p) \) denote the group of all orientation-preserving self-homeomorphisms of \( \Sigma_2 \) fixing the point \( p \), and let \( \Homeo_0(\Sigma_2, p) \) be the subgroup of \( \Homeo^+(\gS_2,p) \) consisting of homeomorphisms isotopic to the identity through isotopies that fix the point \( p \) at every stage.

\begin{definition}
The \emph{mapping class group} of the marked surface \( (\gS_2, p) \) is the group 
\[
\Mod(\Sigma_2,p) = \Homeo^+(\Sigma_2,p)/\Homeo_0(\Sigma_2,p)
\]
\end{definition}

By keeping track of the point \( p \), we obtain a natural action of \( \Mod(\gS_2,p) \) on \( \pi_1(\gS_2, p) \).
In fact, this action induces an isomorphism from \( \Mod(\gS_2,p) \) to an index-2 subgroup of \( \Aut(\pi_1(\gS_2,p)) \) (see \cite[Chapter 8]{fm}).

Given a simple closed curve \( c \) on \( \gS_2 \), we let \( T_c \in \Mod(\gS_2,p) \) denote the isotopy class of a right Dehn twist about \( c \).
The five Dehn twists \( T_a, T_b, T_c, T_d, T_e \) about the curves \( a,b,c,d,e \) as shown in Figure \ref{fig:t2twists} generate \( \Mod(\gS_2, p) \) (see \cite[Theorem 4.13]{fm}).
In the proof of Theorem \ref{lem:evenk}, we will need explicit computations of the action of mapping classes on elements of the fundamental group; the action of the generators is given in the lemma below, whose proof we leave as an exercise for the reader.

\begin{lemma}\label{t:twistformulas} 
Let $w,x,y,z$ be the standard generators of $\pi_1(\Sigma_2, p)$ shown in Figure~\ref{fig:t2twists}.
If $T_a, T_b, T_c, T_d, T_e$ are the Dehn twists about the curves \( a,b,c,d,e \) as in Figure~\ref{fig:t2twists}, then

\begin{multicols}{2}
\begin{itemize}
\item      $T_a: x \mapsto z^{-1}x$  
\item       $T_b: z \mapsto xz$  
\item        $T_c: x \mapsto xwz\inv$,  $\ y\mapsto ywz\inv$ 
\item       $T_d: w \mapsto y\inv w$   
\item        $T_e: y \mapsto wy$

\item      $T_{a}\inv: x \mapsto zx$  
\item       $T_{b}\inv: z \mapsto x^{-1}z$  
\item        $T_{c}\inv: x \mapsto xzw\inv$,  $\ y\mapsto yzw\inv$ 
\item       $T_{d}\inv: w \mapsto yw$   
\item        $T_{e}\inv: y \mapsto w\inv y$
\end{itemize}
\end{multicols}
where an automorphism fixes a generator of \( \pi_1(\gS_2,p) \) if it is not listed.
\end{lemma}

\section{Main result}
\label{s:construction}

A \emph{handlebody} is a compact, orientable, irreducible 3-manifold with a nonempty connected boundary whose inclusion is \( \pi_1 \)-surjective.
The \emph{genus} of a handlebody \( H \) is defined to be the genus of \( \partial H \).
We will focus on the genus two handlebody, which one can visualize as the compact 3-manifold bounded by an embedded copy of a genus two surface in \( \rr^3 \) (such an embedding is drawn in Figure \ref{fig:t2twists}).
 
Given a simple closed curve \( \gamma \) in the boundary of a genus two handlebody \( H \), we define the 3-manifold \( M_\gamma \) to be
\[
M_\gamma = H\,  \bigcup_{\nu(\gamma)}\,  (D^2 \times [0,1]),
\]
where \( D^2 \) denotes the 2-disk, \( \nu(\gamma) \) is a regular neighborhood of \( \gamma \) in \( \partial H \) homeomorphic to \( \gamma \times [0,1] \), and \( \nu(\gamma) \) is identified with \(\partial D^2 \times [0,1] \).
Colloquially, \( M_\gamma \) is the manifold obtained by gluing a 2-handle onto \( \gamma \).

\begin{lemma}
Let $\gamma$ be a simple closed curve on the boundary of a genus two handlebody $H$.
If \( \gamma \) is separating (that is, \( \partial H \ssm \gamma \) is disconnected), then $b_1(M_{\gamma}) = 2$, where \( b_1 \) denotes the first Betti number.
\end{lemma} 

\begin{proof}
We can deformation retract $M_{\gamma}$ to a complex with two $1$-cells and one $2$-cell. Therefore, $b_1 = 1$ or $2$. Let $x$ and $y$ denote the generators of $\pi_1(H)$, and $x,y,z,w$ the generators of $\pi_1(\partial H)$. The attaching map of the two cell comes from removing the letters $z$ and $w$ from the word that $\gamma$ represents in $\pi_1(\partial H)$. 

If $\gamma$ is separating, then $\gamma$ bounds a subsurface of $\partial H$, and hence $[\gamma] \in [\pi_1(\partial H),\pi_1(\partial H)]$.
Therefore, the total sum of the exponents of $x$ and $y$ in $\gamma$  is zero, so the boundary map of the $2$-cell is zero on homology. 
\end{proof}

We now give the construction we will be using in the proof of the Main Theorem.
For the remainder of the section, we fix \( H \) to be a genus two handlebody, \( p \) to be a point in \( \partial H \), \( w,x,y,z \) to be the standard generators of \( \pi_1(\partial H, p) \) shown in Figure \ref{fig:t2twists}, and \( a,b,c,d,e \) to be the simple closed curves on \( \partial H \) as shown in Figure \ref{fig:t2twists}.
In addition, fix a simple, separating closed curve \( \gamma \) on \( \partial H \) representing the commutator \( [x,z] \in \pi_1(\partial H) \).

\begin{definition}
Given an element \( g \in \Mod(\partial H, p) \), define \( M_g \) to be the manifold
\[
M_g = M_{g(\gamma)}.
\]
\end{definition}

\begin{Theorem}
\label{lem:evenk}
For each $n \in \mathbb{N}$ let $g_n = T_{b}^{-1}(T_{d}^{-1}T_c)^{n+1}$.  The polygon $\mathcal{M}_{M_{g_n}}$ has $4n$ marked vertices. 
\end{Theorem}

\begin{Remark}
Here are the first two relators that this algorithm constructs (capital letters denote inverses). 

\begin{enumerate}
\item xyxyyxyXYXYYXY

\item xyxyyxyxyyxyyxyxyyxyXYXYYXYXYYXYYXYXYYXY

\end{enumerate}

Successive relators have a similar structure; namely, the first half of the relator consists of a word containing only positive powers of \( x \) and \( y \) with the second half obtained from the first by replacing each instance of \( x \) and \( y \) with their inverses.
\end{Remark}

\begin{proof}[Proof of Theorem \ref{lem:evenk}]

Let $g = T_b^{-1}T_d^{-1}T_cT_b$, so that $g_{n+1}=g\cdot g_n$. We will inductively define two sequences of words in the generators of $\pi_1(\partial H, p)$: 
\begin{align*}
    A_1 = ywz\inv x  \hspace{0.3in}  & A_{n+1} = B_nA_n\\
    B_1 = y^2wz\inv x  \hspace{0.25in} & B_{n+1} = B_nB_nA_n = B_nA_{n+1}
\end{align*}
We claim that 
\begin{equation}
    A_{n+1} = gA_n \quad \text{ and } \quad  B_{n+1} = gB_n.
\label{abstep}
\end{equation}
We proceed by induction.
The base case is established by a straightforward computation: 
\[
gA_1 = B_1A_1 = A_2
\quad \text{ and } \quad
gB_1 = B_1B_1A_1 = B_2.
\]
Now let \( n \in \bn \) and assume \( gA_{n-1}= A_n \) and \( gB_{n-1} = B_n \).
We then have that
\[
gA_n = g(B_{n-1})g(A_{n-1}) = B_nA_n = A_{n+1}
\]
and
\[
gB_n = g(B_{n-1})g(B_{n-1})g(A_{n-1}) = B_nB_nA_n = B_{n+1}.
\]

Given any word $A$ in $\{w,x,y,z\}$, let $\bar A$ denote the word obtained by eliminating the letters $w$ and $z$, and then reducing and cyclically reducing.
The images of \( x \) and \( y \) generate \( \pi_1(H, p ) \), and the word $\bar A$ determines a walk in $H_1(M_v, \zz)  \cong \zz^2$ as in Section \ref{s:algo}.
 
Let $w(\bar A_n)$ and \( w(\bar B_n) \) denote the width of these walks and \( h(\bar A_n) \) and \( h(\bar B_n)\) their height. 
We have  $$w(\bar A_1) = 1, h(\bar A_1) = 1, w(\bar B_1) = 1, h(\bar B_1)= 2$$ and it follows from the observation that \( A_n \) and \( B_n \) contain only positive powers of \( x \) and \( y \) that 
\begin{align*}
w(\bar A_{n+1}) &= w(\bar B_n) + w(\bar A_n) \\
h(\bar A_{n+1}) &= h(\bar B_n) + h(\bar A_n) \\
w(\bar B_{n+1}) &= w(\bar B_n) + w(\bar A_{n+1}) \\
h(\bar B_{n+1}) &=h(\bar B_n) + h(\bar A_{n+1}).
\end{align*}
Let $f_n$ denote the $n^{th}$ term of the Fibonacci sequence starting with $f_0 = f_1 = 1$. 
Another short induction argument yields
\begin{equation}\label{wh}
w(\bar A_n) = f_{2n-2}, \,\, h(\bar A_n) = f_{2n-1}, \,\, w(\bar B_n) = f_{2n-1}, \,\, \text{and} \,\, h(\bar B_n) = f_{2n}
\end{equation}

We claim that
\begin{equation}
    g_n([x,z]) = xA_1A_2\ldots A_nB_nB_{n-1}\ldots B_1yw(g_n(x^{-1}z^{-1}))
\end{equation}

This can easily be checked for $n=1$ and follows inductively from \eqref{abstep} and the fact that
\[
g(x) = xA_1 \quad \text{and} \quad g(yw) = B_1yw.
\]

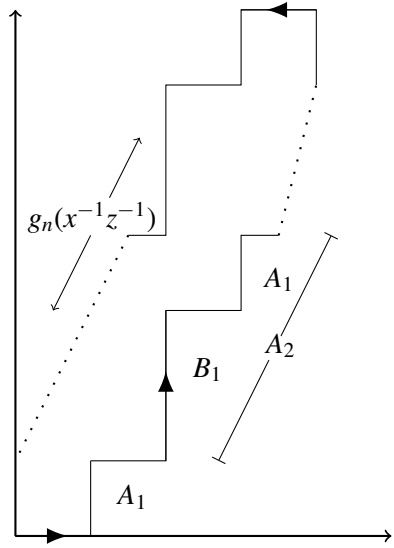
\begin{figure}
    \centering
    \begin{tikzpicture}
    \draw[thick,->] (0,0) -- (5,0); 
    \draw[thick,->] (0,0) -- (0,7); 
    \draw
           (1, 0) 
        -- (0, 0);
    \draw[postaction={decorate}]
    	   (0, 0) 
        -- (1, 0) 
        -- (1, 1) node[below right, xshift=0.2cm, yshift=-0.2cm] {$A_1$};
     \draw[postaction={decorate}]
           (1, 1)
        -- (2, 1)
        -- (2, 3) node[below right, xshift=0.2cm, yshift=-0.5cm] {$B_1$}
        -- (3, 3)
        -- (3, 4) node[below right, xshift=0.15cm, yshift=-0.3cm] {$A_1$}
        -- (3.5, 4);
    \draw[|-] 
           (2.7,1) 
        -- (3.37,2.34) node[above right, xshift=-0.2cm, yshift=-0.1cm] {$A_2$}; 
    \draw[-|] (3.53, 2.66) -- (4.2,4);
    \draw[thick, loosely dotted] 
           (3.5, 4) 
        -- (4, 6);
    \draw[postaction={decorate}]
           (4, 6) 
        -- (4, 7) 
        -- (3, 7) 
        -- (3, 6) 
        -- (2, 6) 
        -- (2, 4) 
        -- (1.5, 4);
    \draw[thick, loosely dotted] 
           (1.5, 4) 
        -- (0, 1);
    \draw[<-] 
           (0.5,3) 
        -- (1,4) node[above right, xshift=-1cm, yshift=-0.1cm] {$g_n(x^{-1}z^{-1})$}; 
    \draw[->] (1.25, 4.5) -- (1.65,5.3);

	\draw[decoration={markings,mark=at position 0.7 with {\arrow[scale=2]{latex}}}, postaction={decorate}] (0,0)--(1,0);
	\draw[decoration={markings,mark=at position 0.6 with {\arrow[scale=2]{latex}}}, postaction={decorate}] (2,1)--(2,3);
	\draw[decoration={markings,mark=at position 0.67 with {\arrow[scale=2]{latex}}}, postaction={decorate}] (4,7)--(3,7);

\end{tikzpicture}
    \caption{A general picture of $g_n([x,z])$, after removing the letters $z$ and $w$. The relator starts at the origin and climbs up the bottom ``staircase'' which is $g_n(xz)$. Then $g_n(x^{-1}z^{-1})$ travels back down.}
 \label{fig:oddk_ti}
    \end{figure}

Fix \( n \in \bn \) and define the list of points \( P_n \)  in $\zz^2$ as follows:
\begin{align*}
P_n = \{  & (1,0), (w(\bar A_1), h(\bar A_1)), \ldots, (w(\bar A_n), h(\bar A_n)), \\
		& (w(\bar B_n), h(\bar B_n)), \ldots, (w(\bar B_1), h(\bar B_1)), (0,1) \} \\
	= \{  & (1,0), (f_0,f_1), (f_2,f_3),\ldots, (f_{2n-2}, f_{2n-1}), \\
		& (f_{2n-1}, f_{2n}), (f_{2n-3}, f_{2n-1}),\ldots, (f_1, f_2), (0,1)\}.
\end{align*}
Note that \( |P_n| = 2n+2 \).
For \( k \in \{0, 1, \ldots, |P_n|\} \),  define the point \( p_k \) in $\zz^2$ by setting \( p_0 = (0,0) \) and
\[
   p_{k} = p_{k-1} + P_n(k) 
\]
for \( 0<k \leq |P_n|\), where \( P_n(k) \) denotes the \( k^{th} \) element in the list \( P_n \). 
Additionally, we define \( \ell_k \) to be the line segment connecting consecutive points, that is, for \( k \in \{1, \ldots, |P_n|\} \) we set
\[
   \ell_k = \overline{p_{k-1}p_k}.
\]

It follows from the construction that for each \( k \in \{0, \ldots, |P_n|\} \) the point \( p_k \) lies on the walk determined by
\[
g_n(xz) = xA_1A_2\ldots A_nB_nB_{n-1}\ldots B_1yw.
\]
Moreover, for \( k \in \{1, \ldots, n+1 \} \), the point \( p_{k+1} \) is the endpoint of the walk determined by \( xA_1\ldots A_k \).
Similarly, for \( k \in \{n+2, \ldots, 2n+1\} \), the point \( p_{k+1} \) is the end point of the walk determined by \( xA_1\ldots A_nB_n \ldots B_{2n-k+2} \). 
We now want to prove that the walk determined by \( g_n(xz) \) is bounded on the right by \( \bigcup_{k} \ell_k \).

We again proceed by induction on both sequences:
As our base case we observe that the vertices in the walk determined by \( \bar A_1 \) are \( (0,0), (0,1) \) and \( (1,1) \); hence the walk is bounded to the right by the line connecting the endpoints of the walk.
Similarly one checks this for \( \bar B_1 \). 
Now the walk determined by \( \bar A_k \) is the concatenation of the walk determined by \( \bar B_{k-1} \) followed by that of \( \bar A_{k-1} \).
Now it follows from \eqref{wh} that the slope of the line connecting the endpoints of the walk corresponding to \( \bar B_{k-1} \) is \( \frac{f_{2k-2}}{f_{2k-3}} \) and the slope connecting the endpoints of the walk given by \( \bar A_k \) is \( \frac{f_{2k-1}}{f_{2k-2}} \).
Standard properties of the Fibonacci sequence guarantee the latter is always smaller.
If we now assume that all the vertices in the walks \( \bar B_{k-1} \) and \( \bar A_{k-1} \) are bounded to the right by the line segments connecting their respective endpoints, then the same is true for the walk given by \( \bar A_k \).
A similar argument can be made for \( \bar B_k \). 
As the walk given by \( g_n(xz) \) is a concatenation of the walks given by the \( \bar A_k \) and \( \bar B_k \)'s, we see that \( \bigcup_k \ell_k \) bounds this walk on the right.

Let \( \Theta_n \co \rr^2 \to \rr^2 \) be defined by \( \Theta_n(v) = -v + p_{2n+2} \).
We now claim that the convex hull \( \mathcal{C}_n \) of the walk given by \( g_n([x,z]) \) has vertex set
\[
V_n = \{p_k : 0\leq k \leq 2n+2 \} \cup \{\Theta_n(p_k): 0< k < 2n+2 \}
\]
and edge set
\[
E_n = \{\ell_k : 1\leq k \leq 2n+2\} \cup \{\Theta_n(\ell_k): 1\leq k \leq 2n+2 \}.
\]
Let \( \hat \ell_k \) be the complete line containing the segment \( \ell_k \).
Again using basic properties of the Fibonacci sequence, the slope of \( \hat \ell_k \) is strictly greater than that of \( \hat\ell_{k+1} \) implying that each \( \hat\ell_k \) bounds \( \mathcal{C}_n \) from the right.
Now combining this with the fact that \( \ell_k \) is contained in \( \mathcal{C}_n \), we must have that \( \ell_k \) is an edge of \( \mathcal{C}_n \).
As the walk is invariant under the transformation \( \Theta_n \), we see that \( \Theta_n(\ell_k) \) is also an edge of \( \mathcal{C}_n \).
Finally, the pointwise union of the elements of \( E_n \) bounds a convex polygon and hence \( E_n \) contains all the edges of \( \mathcal{C}_n \).

\begin{figure}[t]
    \centering
    \begin{tikzpicture}[scale = .8]
  
    \draw[thick,->] (0,0) -- (6,0); 
    \draw[thick,->] (0,0) -- (0,10.3); 

    \draw[gray, thick, loosely dotted] 
           (1, 0) 
        -- (0, 0) 
        -- (1, 0) 
        -- (1, 1) 
        -- (2, 1)
        -- (2, 3) 
        -- (3, 3)
        -- (3, 4) 
        -- (4, 4);
    \draw[gray, thick, loosely dotted] 
           (5, 9.3) 
        -- (5, 10.3) 
        -- (4, 10.3) 
        -- (4, 9.3) 
        -- (3, 9.3) 
        -- (3, 7.3)
        -- (2, 7.3)
        -- (2, 6.3)
        -- (1, 6.3);

    \draw
           (0, 1) node[left] {\footnotesize $\Theta(p_{2n+2})$} node {\color{blue}\textbullet}
        -- node [left] {\footnotesize $\Theta(l_{2n+2})$} (0, 0) node[below left] {\footnotesize $p_{0}$} node {\color{blue}\textbullet}
        -- node[below] {\footnotesize $l_1$} (1, 0) node[below right] {\footnotesize $p_1$} node {\color{blue}\textbullet}
        -- node[below right] {\footnotesize $l_2$} (2, 1) node[below right] {\footnotesize $p_2$} node {\color{blue}\textbullet}
        -- node[below right] {\footnotesize $l_3$} (4, 4) node[below right] {\footnotesize $p_3$} node {\color{blue}\textbullet};
    \draw[thick, loosely dotted] 
           (4, 4) 
        -- (5, 9.3);
    \draw
           (5, 9.3) node[below right] {\footnotesize $p_{2n+1}$} node {\color{blue}\textbullet}
        -- node[right] {\footnotesize $l_{2n+2}$} (5, 10.3) node[above right] {\footnotesize $p_{2n+2}$} node {\color{blue}\textbullet}
        -- node[above] {\footnotesize $\Theta(l_1)$} (4, 10.3) node[above left] {\footnotesize $\Theta(p_1)$} node {\color{blue}\textbullet}
        -- node[left] {\footnotesize $\Theta(l_2)$} (3, 9.3) node[left] {\footnotesize $\Theta(p_2)$} node {\color{blue}\textbullet}
        -- node[above left] {\footnotesize $\Theta(l_3)$} (1, 6.3) node[above left] {\footnotesize $\Theta(p_3)$} node {\color{blue}\textbullet};

    \draw[thick, loosely dotted] 
           (1, 6.3) 
        -- (0, 1);

    \draw[red]
           (0.7, 2)
        -- (0.5, 0.5) node[below left] {}
        -- (1.5, 1.5) node[below right] {}
        -- (3.5, 4.5) node[below right] {}
        -- (3.75, 6);
    \draw[red, thick, loosely dotted]
           (3.75, 6)
        -- (4.4, 9.3);
    \draw[red]
           (4.4, 9.3)
        -- (4.5, 9.8) node[above right] {}
        -- (3.5, 8.8) node[above left] {}
        -- (1.5, 5.8) node[above left] {}
        -- (1.4, 5.3);
    \draw[red, thick, loosely dotted]
           (1.4, 5.3)
        -- (0.66, 2);

    \draw[gray, decoration={markings,mark=at position 1 with {\arrow[scale=2]{>}}},postaction={decorate}] (0,0) -- (0.5,0.5);
    \draw[gray, decoration={markings,mark=at position 1 with {\arrow[scale=2]{>}}},postaction={decorate}] (0,1) -- (0.5,0.5);
    \draw[gray, decoration={markings,mark=at position 1 with {\arrow[scale=2]{>}}},postaction={decorate}] (1,0) -- (0.5,0.5);
    \draw[gray, decoration={markings,mark=at position 1 with {\arrow[scale=2]{>}}},postaction={decorate}] (2,1) -- (1.5,1.5);
    \draw[gray, decoration={markings,mark=at position 1 with {\arrow[scale=2]{>}}},postaction={decorate}] (4,4) -- (3.5,4.5);
    \draw[gray, decoration={markings,mark=at position 1 with {\arrow[scale=2]{>}}},postaction={decorate}] (5,10.3) -- (4.5,9.8);
    \draw[gray, decoration={markings,mark=at position 1 with {\arrow[scale=2]{>}}},postaction={decorate}] (5,10.3) -- (4.5,9.8);
    \draw[gray, decoration={markings,mark=at position 1 with {\arrow[scale=2]{>}}},postaction={decorate}] (4,10.3) -- (4.5,9.8);
    \draw[gray, decoration={markings,mark=at position 1 with {\arrow[scale=2]{>}}},postaction={decorate}] (3,9.3) -- (3.5,8.8);
    \draw[gray, decoration={markings,mark=at position 1 with {\arrow[scale=2]{>}}},postaction={decorate}] (1,6.3) -- (1.5,5.8);

\end{tikzpicture}
    \caption{The marked polytope $\cM_{M_{g_n}}$ with $4n$ vertices.}
 \label{fig:4ivertices}
\end{figure}
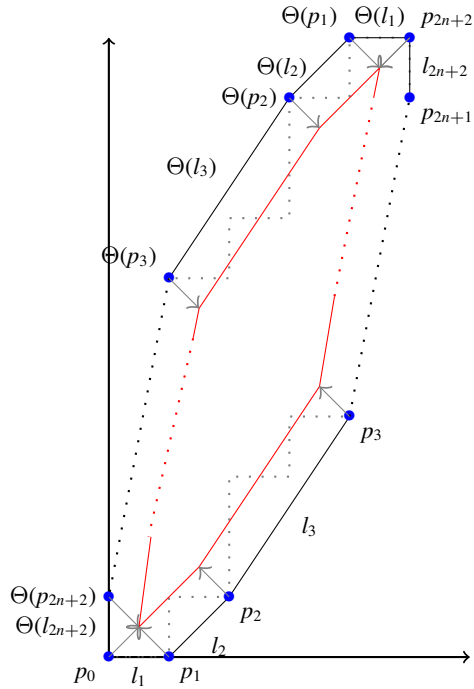

To finish, we note that \( \mathcal{C}_n \) has \( 4n+4 \) vertices and applying the remaining steps of the algorithm presented in Section \ref{s:algo} results in the polygon \( \mathcal{M}_{M_{g_n}} \) with \( 4n \) vertices.
In particular, the two collections of vertices \( \{\Theta_n(p_{2n+1}), p_0, p_1 \} \) and \( \{p_{2n+1}, p_{2n+2}, \Theta_n(p_1)\} \) each collapse to a single vertex in the process of obtaining \( \mathcal{M}_{M_{g_n}} \) from \( \mathcal{C}_n \) (see Figure \ref{fig:4ivertices}). It is easy to see that two vertices $v$ and $v'$ of the convex hull collapse to a single vertex in $\cM_\pi$ only if they lie on the same square, so no other vertices of the convex hull collapse. 
\end{proof}

Similar arguments give the following theorem; we only sketch the proof.
\begin{Theorem}
\label{lem:oddk}
For each $n\in\mathbb{N}$ let $h_n = T_b^{-2}(T_d^{-1}T_c)^{n+1}$.
The polygon $\mathcal{M}_{M_{h_n}}$ has $4n+2$ marked vertices. 
\end{Theorem}

\begin{proof}
With $g_n, A_k, B_k$ as in the proof of Theorem ~\ref{lem:evenk}, the composition of $T_b\inv$ with $g_n$ simply increases the width of $A_1,B_1$ by $1$. Therefore, the Fibonacci sequences start at $f_0 = f_1 = 2$ and thus doubles the width of $A_k, B_k$.  Exactly the same argument as above applies to obtain a convex hull of $4n+4$ vertices, which yields a polygon of $4n+2$ vertices.
\end{proof}

\begin{Remark}
We do not know if the $3$-manifolds that we construct have any alternative descriptions. For example, it would be interesting if these manifolds were all link complements that had a uniform construction. In \cite{L06}, Licata used Heegaard Floer homology to compute $\cM_\pi$ for all two-component, four-strand pretzel links, and the unit balls are shaped similarly to the polytopes we construct. On the other hand, all of Licata's examples have 8 vertices, so cannot match our examples. 
\end{Remark}

\section{Appendix: Baumslag-Solitar fundamental groups}\label{s:appendix}

We study in detail a specific collection of $3$-manifolds that were helpful for us in understanding the Thurston norm of tunnel number one manifolds. 
In particular, we look at 3-manifolds whose fundamental group is isomorphic to a Baumslag-Solitar group of the form \( \mathrm{BS(m,m)} \) for some positive integer \( m \), that is, 
\[
\pi_1 = \langle x,y\, | \, xy^{m}x^{-1}y^{-m}\rangle.
\]
For such a manifold, the Thurston norm assigns 0 to the cohomology class dual to $x \in H_1(M)$, and $m-1$ to the class dual to $y$.
This can be seen using the algorithm presented in Section \ref{s:algo}.
 
In the notation of Section \ref{s:construction}, if for a positive integer \( m \) we define \( f_m \in \Mod(\partial H) \) to be
\[
f_m = T_d^mT_c\inv T_d T_c,
\]
then \( \pi_1(M_{f_m}) = \mathrm{BS}(m,m) \). 

For the purpose of this appendix, we give another useful construction of the manifold \( M_{f_m} \):
Let \( M_1 = \mathbb{S}^1 \times \mathbb{D} \) be a solid torus, \( (p,q) \in \partial M_1 \), \( \lambda = \mathbb{S}^1 \times \{q\} \), and \( \mu = \{p\} \times \partial \mathbb{D} \). 
Let \( \gamma\) be a curve in \( \partial M_1 \) representing the homology class \( m[\lambda] + [\mu] \) in \( H_1(\partial M_1) \) and let \( \gamma_0 \) and \( \gamma_1\) be distinct parallel copies of \( \gamma \), that is, \( \gamma, \gamma_0, \) and \( \gamma_1 \) are pairwise disjoint and homologous.
Now, let \( M_2 = \mathbb{S}^1 \times [0,1]\times [0,1] \).

Let \( \nu(\gamma_0) \) and \( \nu(\gamma_1) \) be disjoint regular neighborhoods of \( \gamma_0 \) and \( \gamma_1 \), respectively.
Let \( \gamma_i^+ \) and \( \gamma_i^- \) denote the boundary components of \( \nu(\gamma_i) \) labelled such that \( \gamma_0 \) and \( \gamma_1 \) are contained in distinct components of \( \partial M_1 \ssm ( \gamma_0^+ \cup \gamma_1^+) \).
Let \( N \) be the manifold obtained by identifying \( \nu(\gamma_i) \) in \( \partial M_1 \) with the annulus \( \mathbb{S}^1 \times [0,1] \times \{i\} \subset \partial M_2 \) such that 
\begin{itemize}
\item \( \gamma_i^+ \) is identified with \( \mathbb{S}^1 \times \{1-i\} \times \{i\} \) and
\item \( \gamma_i^- \) is identified with \( \mathbb{S}^1 \times \{i\} \times \{i\} \).
\end{itemize}

Up to homotopy, we may assume that \( \gamma \) is contained in the annulus \( A \) bounded by \( \gamma_0^+ \) and \( \gamma_1^- \) in \( \partial M_1 \).
Pick a point \( t \in \mathbb{S}^1 \) and let \( \delta \) be the loop obtained by closing up the segment \( \{t\} \times \{1\} \times [0,1] \) in \(\partial M_2 \) with an arc contained in \( A \) and intersecting \( \gamma \) once. 
We can then see that \( \gamma \) and \( \delta \) are contained a torus component of \( \partial N \) and hence commute as elements of the fundamental group.
Furthermore, an application of Van Kampen's theorem shows that the curves \( \lambda \) and \( \delta \) generate \( \pi_1(N) \) and yields the presentation \( \pi_1(N) = \langle \delta, \lambda \, | \, \delta\lambda^m\delta\inv\lambda^{-m} \rangle \).
This construction is shown in Figure \ref{bs-group}.

\begin{figure}
\centering
\includegraphics[trim={0.85cm 15.55cm 6.4cm 1.5cm}, clip, scale=.65]{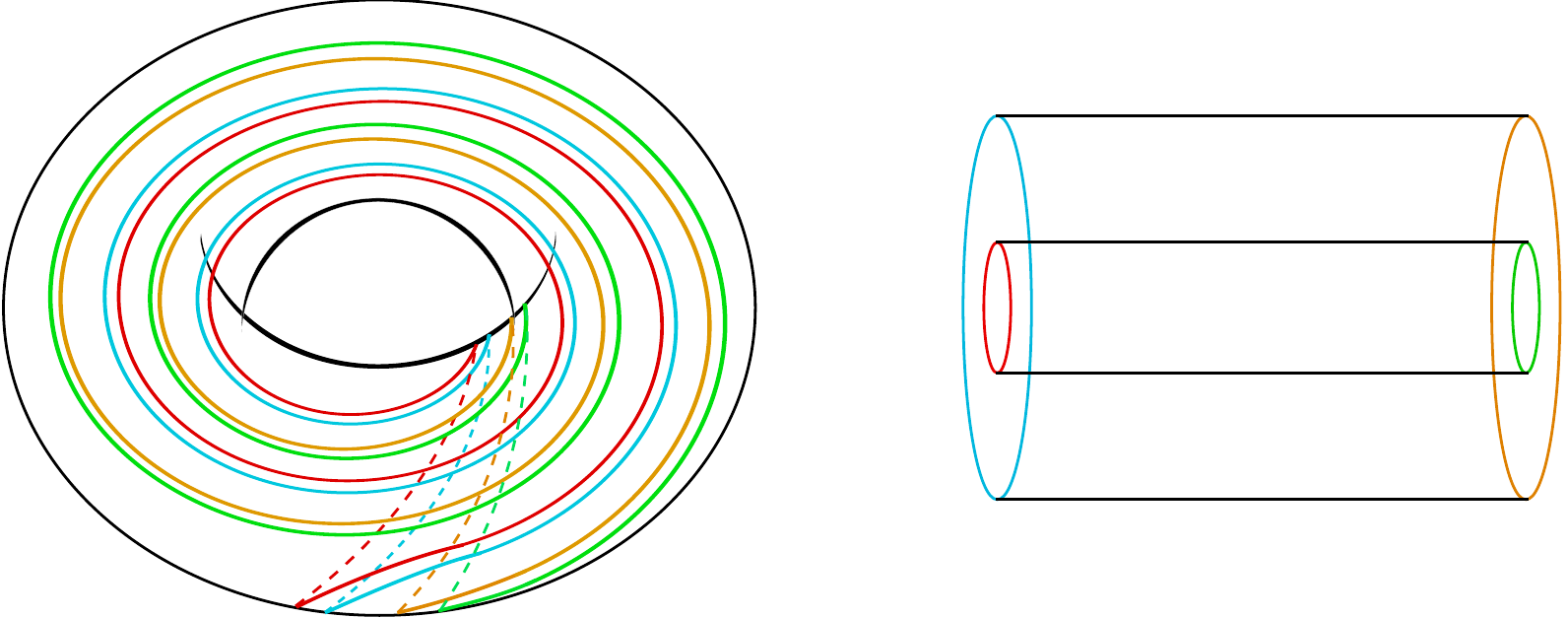}
\caption{Shown here is the manifold \( M_{f_m}\cong N = M_1 \cup M_2 \) with \(  m = 4 \). 
The manifold \( M_1 \) is drawn on the left and \( M _2 \) on the right.  Curves of the same color are identified in \( M_{f_m} \).
}
\label{bs-group}
\end{figure}

Now, it is easy to see that \( \mathrm{BS}(m,m) \) has an infinite and normal cyclic subgroup (e.g. take the group generated by \( \lambda^m \) in \( \pi_1(N) \)).
The Seifert fibered space theorem (see \cite{CJ94, G92}) implies that both \( M_{f_m} \) and \( N \) are Seifert fibered. 
In \cite[Theorem 3.1]{scott}, Scott showed that if two Seifert fibered spaces have infinite and isomorphic fundamental groups, then they are in fact homeomorphic; hence, \( N \) and \( M_{f_m} \) are homeomorphic.

Using this new description of \( M_{f_m} \) we can find the minimal dual surfaces to \( \lambda \) and \( \delta \).
By applying the algorithm in Section \ref{s:algo}, we know there exists a surface of Euler characteristic 0 dual to \( \delta \), e.g. the surface \( \mathbb{S}^1 \times [0,1] \times \{\frac 12\} \subset M_2 \).

\begin{figure}[t]
\centering
	\includegraphics[scale=0.37, trim={0.8cm 5.15cm 0.5cm 3cm}, clip]{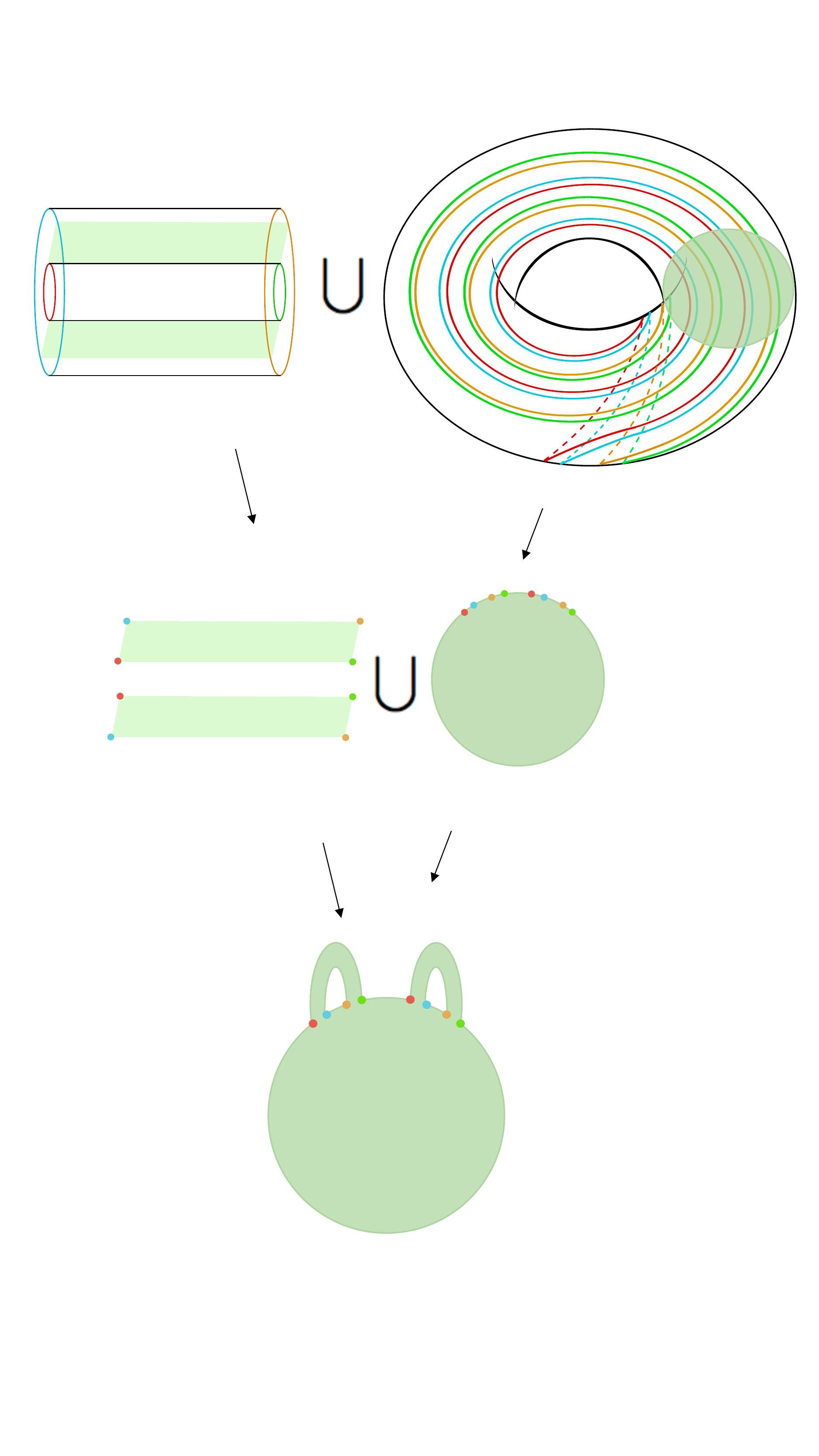}
\caption{The dual disk to \( \lambda \) in the solid torus, and its extension to a dual surface in \( N \).
}
\label{y-dual-pre}
\end{figure}

Let us now focus on \( \lambda \).
In \( M_2 \), \( \lambda \) is dual to the disk \( D \) bounded by \( \mu \); however, \( \mu \) does not live in the boundary of \( M \); we will fix this with a surgery.
For \( i \in \{0,1\} \), the intersection \( \nu(\gamma_i) \cap \mu \) can be written as a disjoint union of \( m \) intervals, call them \( I_1^i, \ldots, I_m^i \) (recall that \( |\gamma_0 \cap \mu| = m \) by construction).
The intervals \( I_k^0 \) and \( I_k^1 \) bound a rectangle \( R_k \) in \( M_2 \) of the form \( I_k^0 \times [0,1] \) for each \( k \in \{1, \ldots, m\} \). 
Now the surface
\[
\Sigma = D \cup R_1 \cup \cdots \cup R_m
\]
is obtained from a disk by attaching \( m \) rectangular strips, so it is a genus-0 surface with \( m \) boundary components.
In particular, \( \chi(\Sigma) = 1-m \).
Furthermore, \( \Sigma \) is dual to \( \lambda \).
The algorithm guarantees this dual surface is minimal.

\end{document}